\documentclass[11pt]{amsart}
\usepackage{mathtools}
\usepackage{tikz}
\usetikzlibrary{arrows.meta,bending,arrows}

\makeatletter
\@namedef{subjclassname@2020}{\textup{2020} Mathematics Subject Classification}
\makeatother

\theoremstyle{plain}
\newtheorem{thm}{Theorem}

\theoremstyle{definition}

\newtheorem{example}{Example}

\newcommand{\vcirc}[1]{\node[draw,circle,inner sep=.6mm,line width=.5pt] at (#1){};}
\newcommand{\vlabel}[1]{\bigg\updownarrow\rlap{#1}}
\newcommand{\meq}[2][0mm]{\raisebox{-.88mm}{\hspace{#1}\ensuremath{#2}\hspace{#1}}}

\tikzset{overcross/.style={double,line width=1.5,white,double=black,double distance=1}}
\tikzset{knotarrow/.style={line width=.6,arrows=-Latex[bend]}}
\tikzset{chordarrow/.style={arrows=-angle 60}}

\begin{document}

\title{Reidemeister Moves in Gauss Diagrams}

\author[Ganzell]{Sandy Ganzell}
\address{Department of Mathematics and Computer Science, St.\ Mary's College of Maryland, 18952 E.~ Fisher Rd., St.\ Mary's City, MD 20686}
\email{sganzell@smcm.edu}

\author[Lehet]{Ellen Lehet}
\address{100 Malloy Hall, Philosophy Department, University of Notre Dame, Notre Dame, IN, 46556}
\email{elehet@nd.edu}

\author[Lopez]{Cristina Lopez}
\address{930 E.\ 87th Pl., Los Angeles, CA 90002}
\email{cristinalopez956@gmail.com}

\author[Magallon]{Gilbert Magallon}
\address{651 Midrock Cors, Mountain View, CA, 94043}
\email{gilbertmagallon@gmail.com}

\author[Thompson]{Alyson Thompson}
\address{Department of Mathematics and Computer Science, St.\ Mary's College of Maryland, 18952 E.~ Fisher Rd., St. Mary's City, MD 20686}
\email{athompson1@smcm.edu}

\thanks{Partially supported by NSF Grant DMS 1005046}

\begin{abstract}
We provide a simple algorithm for recognizing and performing Reidemeister moves in a Gauss diagram.
\end{abstract}
\maketitle

\section{Introduction}\label{intro}

In \cite{MR1721925}, Kauffman introduced the theory of virtual knots. Like the classical theory, virtual knot theory has a useful diagrammatic approach. Virtual knot diagrams can be thought of as closed curves in the plane with regular crossings, having extra structure at the crossings. In the classical theory, this extra structure is indicated by overcrossings and undercrossings. In the virtual theory, a third kind of crossing is allowed, namely virtual crossings, which are indicated in the diagrams with a small circle around the vertex.

Equivalence of virtual knots may be defined by means of a set of local moves (the extended Reidemeister moves) on their diagrams. Figure \ref{VRmoves} illustrates the extended Reidemeister moves: the classical Reidemeister moves R1, R2, R3, the virtual moves V1, V2, V3 and the semivirtual move SV.
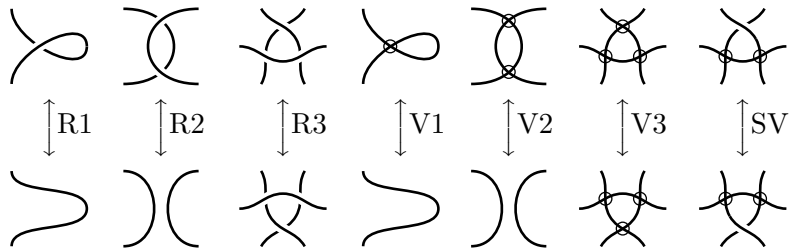
\begin{figure}[htp]
\centering
\begin{tikzpicture}[line width=1]
\matrix[ampersand replacement=\&,column sep=4mm]{
\draw(-.5,.5) to[out=-90,in=-90] (.5,0);
\draw[overcross] (.5,0) to[out=90,in=90] (-.5,-.5); \&
\draw[looseness=2.3] (-.5,-.5) to[out=0,in=0] (-.5,.5);
\draw[looseness=2.3,overcross] (.5,-.5) to[out=180,in=180] (.5,.5); \&
\draw (-120:.58) to[out=60,in=-120] (150:.2) to[out=60,in=-120] (60:.58);
\draw[rotate=-120,overcross] (-120:.58) to[out=60,in=-120] (150:.2) to[out=60,in=-120] (60:.58);
\draw[rotate=120,overcross] (-120:.58) to[out=60,in=-120] (150:.2) to[out=60,in=-120] (60:.58); \&
\draw(-.5,.5) to[out=-90,in=-90] (.5,0) to[out=90,in=90] (-.5,-.5); \vcirc{-.14,0}
\& \draw[looseness=2.3] (-.5,-.5) to[out=0,in=0] (-.5,.5);
\draw[looseness=2.3] (.5,-.5) to[out=180,in=180] (.5,.5); \vcirc{0,.34}\vcirc{0,-.34} \&
\draw (-120:.58) to[out=60,in=-120] (150:.2) to[out=60,in=-120] (60:.58);
\draw[rotate=-120] (-120:.58) to[out=60,in=-120] (150:.2) to[out=60,in=-120] (60:.58);
\draw[rotate=120] (-120:.58) to[out=60,in=-120] (150:.2) to[out=60,in=-120] (60:.58); \vcirc{90:.265}\vcirc{-150:.265}\vcirc{-30:.265}
\& \draw (-120:.58) to[out=60,in=-120] (150:.2) to[out=60,in=-120] (60:.58);
\draw[rotate=120] (-120:.58) to[out=60,in=-120] (150:.2) to[out=60,in=-120] (60:.58);
\draw[rotate=-120,overcross] (-120:.58) to[out=60,in=-120] (150:.2);\draw[rotate=-120](150:.2) to[out=60,in=-120] (60:.58);\vcirc{-150:.265}\vcirc{-30:.265}\\
\node{\vlabel{R1}}; \& \node{\vlabel{R2}}; \& \node{\vlabel{R3}}; \& \node{\vlabel{V1}}; \& \node{\vlabel{V2}}; \& \node{\vlabel{V3}}; \& \node{\vlabel{SV}};\\
\draw[looseness=.8] (-.5,-.5) to[out=90,in=-90] (.5,0) to[out=90,in=-90] (-.5,.5); \& \draw[looseness=1.4] (-.5,-.5) to[out=0,in=0] (-.5,.5);\draw[looseness=1.4] (.5,.5) to[out=180,in=180] (.5,-.5); \& \draw (-120:.58) to[out=60,in=-120] (-30:.2) to[out=60,in=-120] (60:.58);
\draw[rotate=-120,overcross] (-120:.58) to[out=60,in=-120] (-30:.2) to[out=60,in=-120] (60:.58);
\draw[rotate=120,overcross] (-120:.58) to[out=60,in=-120] (-30:.2) to[out=60,in=-120] (60:.58); \& \draw[looseness=.8] (-.5,-.5) to[out=90,in=-90] (.5,0) to[out=90,in=-90] (-.5,.5); \& \draw[looseness=1.4] (-.5,-.5) to[out=0,in=0] (-.5,.5);\draw[looseness=1.4] (.5,.5) to[out=180,in=180] (.5,-.5); \& \draw (-120:.58) to[out=60,in=-120] (-30:.2) to[out=60,in=-120] (60:.58);
\draw[rotate=-120] (-120:.58) to[out=60,in=-120] (-30:.2) to[out=60,in=-120] (60:.58);
\draw[rotate=120] (-120:.58) to[out=60,in=-120] (-30:.2) to[out=60,in=-120] (60:.58); \vcirc{-90:.265}\vcirc{150:.265}\vcirc{30:.265}
\& \draw (-120:.58) to[out=60,in=-120] (-30:.2) to[out=60,in=-120] (60:.58);
\draw[rotate=120] (-120:.58) to[out=60,in=-120] (-30:.2) to[out=60,in=-120] (60:.58);
\draw[rotate=-120] (-120:.58) to[out=60,in=-120] (-30:.2);\draw[rotate=-120,overcross](-30:.2) to[out=60,in=-120] (60:.58);\vcirc{150:.265}\vcirc{30:.265}\\
};
\end{tikzpicture}
\caption{The extended Reidemeister moves.}
\label{VRmoves}
\end{figure}
We may then define a virtual knot to be an equivalence class of virtual diagrams modulo these moves (and planar isotopy).

In \cite{MR1763963} it is proved that two classical knots are equivalent under extended Reidemeister moves if and only if they are equivalent under classical Reidemeister moves. Thus virtual knot theory may be considered a generalization of the classical theory.

Nevertheless, some of our intuition can lead us astray with diagrams that contain virtual crossings. It is important to restrict ourselves to the moves, and not consider physical movements in space.
\begin{figure}[htbp]
\centering
    \begin{tikzpicture}[scale=.6,line width=1pt]
    \draw (3,2) to[out=180,in=90] (2,0) to[out=-90,in=180] (3,-2);
    \draw (-3,-2) to[out=0,in=-90] (-2,0) to[out=90,in=0] (-3,2);
    \draw[overcross,looseness=.8] (4,0) to[out=135,in=45] (-4,0) to[out=-45,in=-135] (4,0);
    \draw (3,-2) to[out=0,in=-45] (4,0) to[out=45,in=0] (3,2);
    \draw (-3,2) to[out=180,in=135] (-4,0) to[out=-135,in=180] (-3,-2);
    \vcirc{-4,0}\vcirc{4,0}
    \end{tikzpicture}\hspace{1cm}
    \begin{tikzpicture}[scale=.6,line width=1pt]
    \draw (3,2) to[out=180,in=90] (2,0) to[out=-90,in=180] (3,-2);
    \draw (-3,-2) to[out=0,in=-90] (-2,0) to[out=90,in=0] (-3,2);
    \draw (2,1) to[out=180,in=0] (0,-.5);
    \draw (0,.5) to[out=180,in=0] (-2,-1);
    \draw[overcross] (2,1) to[out=0,in=135] (4,0) to[out=-135,in=0] (2,-1) to[out=180,in=0] (0,.5);
    \draw[overcross] (0,-.5) to[out=180,in=0] (-2,1) to[out=180,in=45] (-4,0) to[out=-45,in=180] (-2,-1);
    \draw (3,-2) to[out=0,in=-45] (4,0) to[out=45,in=0] (3,2);
    \draw (-3,2) to[out=180,in=135] (-4,0) to[out=-135,in=180] (-3,-2);
    \vcirc{-4,0}\vcirc{4,0}
    \end{tikzpicture}
\caption{Distinct virtual knots.}
\label{Kishino}
\end{figure}
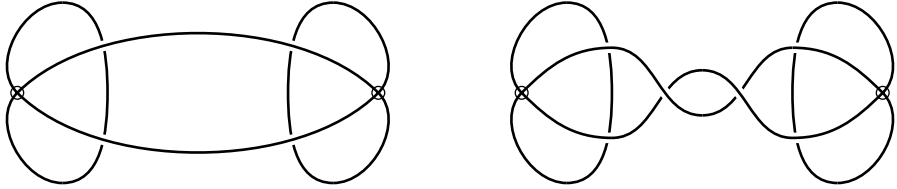
For example, the virtual knots in Figure~\ref{Kishino} are distinct, though they appear to differ only by a physical twist. These knots can be distinguished by the arrow polynomial \cite{MR2583800}.

There are two additional Reidemeister-like moves, known as the \emph{forbidden moves}, illustrated in Figure~\ref{Fmoves}. 
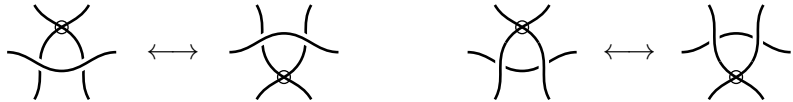
\begin{figure}[htbp]
\centering
\begin{tikzpicture}[line width=1pt,baseline,scale=1.25]
\draw (-120:.58) to[out=60,in=-120] (150:.2) to[out=60,in=-120] (60:.58);
\draw[rotate=-120] (-120:.58) to[out=60,in=-120] (150:.2) to[out=60,in=-120] (60:.58);
\draw[rotate=120,overcross] (-120:.58) to[out=60,in=-120] (150:.2) to[out=60,in=-120] (60:.58);
\vcirc{90:.265}
\end{tikzpicture}
\meq[2mm]{\longleftrightarrow}
\begin{tikzpicture}[line width=1pt,baseline,scale=1.25]
\draw (-120:.58) to[out=60,in=-120] (-30:.2) to[out=60,in=-120] (60:.58);
\draw[rotate=-120] (-120:.58) to[out=60,in=-120] (-30:.2) to[out=60,in=-120] (60:.58);
\draw[rotate=120,overcross] (-120:.58) to[out=60,in=-120] (-30:.2) to[out=60,in=-120] (60:.58);
\vcirc{-90:.265}
\end{tikzpicture}\hspace{1.5cm}
\begin{tikzpicture}[line width=1pt,baseline,scale=1.25]
\draw (150:.2) to[out=60,in=-120] (60:.58);
\draw[rotate=-120] (-120:.58) to[out=60,in=-120] (150:.2);
\draw[rotate=120] (-120:.58) to[out=60,in=-120] (150:.2) to[out=60,in=-120] (60:.58);
\draw[overcross,rotate=-120] (150:.2) to[out=60,in=-120] (60:.58);
\draw[overcross] (-120:.58) to[out=60,in=-120] (150:.2);
\vcirc{90:.265}
\end{tikzpicture}
\meq[2mm]{\longleftrightarrow}
\begin{tikzpicture}[line width=1pt,baseline,scale=1.25]
\draw (-120:.58) to[out=60,in=-120] (-30:.2);
\draw[rotate=-120] (-30:.2) to[out=60,in=-120] (60:.58);
\draw[rotate=120] (-120:.58) to[out=60,in=-120] (-30:.2) to[out=60,in=-120] (60:.58);
\draw[overcross] (-30:.2) to[out=60,in=-120] (60:.58);
\draw[rotate=-120,overcross] (-120:.58) to[out=60,in=-120] (-30:.2);
\vcirc{-90:.265}
\end{tikzpicture}
\caption{Forbidden moves.}
\label{Fmoves}
\end{figure}
Neither of these moves can be obtained as a sequence of extended Reidemeister moves. If we allow both of these moves, then any virtual knot can be transformed into any other virtual knot \cite{MR1840276,AlissaBlakeSandy}, hence the designation of these moves as forbidden. If we allow one forbidden move but not the other, we obtain what are known as \emph{welded knots}, developed by Satoh \cite{MR1758871} and Kamada \cite{MR2351010}.

\section{Gauss Diagrams and Reidemeister Moves}\label{gaussdiagrams}
One motivation for the development of virtual knots was to provide a correspondence between knots and {\em Gauss diagrams}.  Given an oriented (virtual) knot, we define its Gauss diagram as follows:  First label all the classical crossings of the knot.  Then traverse the knot, noting the sequence of crossing labels (so each label appears twice). Write this sequence counterclockwise\footnote{Either direction is acceptable, but we will always draw our diagrams counterclockwise for consistency.} around a circle.  Add a directed chord (indicated by an arrow) to the circle for each crossing, pointing from the label corresponding to the overcrossing to the label corresponding to the undercrossing.  Finally, each arrow is labeled with the sign of the crossing.
\begin{figure}[htbp]
\centering
\begin{tikzpicture}[baseline={([yshift=3pt]current bounding box)},scale=1.25,line width=1pt]
\draw[knotarrow,shorten >=3mm] (-150:1) to[out=-60,in=-120] (-30:.33) to[out=60,in=0] (90:1);
\draw (-150:1) to[out=-60,in=-120] (-30:.33) to[out=60,in=0] (90:1);
\draw (-150:.33) to[out=-60,in=-120] (-30:1);
\draw (90:.33) to[out=180,in=120] (-150:1);
\draw[overcross] (90:1) to[out=180,in=120] (-150:.33);
\draw[overcross] (-30:1) to[out=60,in=0] (90:.33);
\vcirc{0,-.475}
\path (-.6,.45) node {1} -- (.6,.45) node {2};
\end{tikzpicture}\hspace{1cm}
\begin{tikzpicture}[baseline,scale=1,line width=1pt]
\draw (0,0) circle[radius=1];
\draw[knotarrow] ([shift=(80:1)]0,0) arc (80:100:1);
\draw[chordarrow] (120:1) node[above left=-.5mm]{1} -- (-60:1) node[below right=-.5mm]{1};
\draw[chordarrow] (180:1) node[left]{2} -- (0:1) node[right]{2};
\path (.4,.15) node {$-$} -- (.05,-.5) node {$-$};
\end{tikzpicture}
\caption{Gauss diagram of a virtual trefoil.}
\label{Gauss1}
\end{figure}
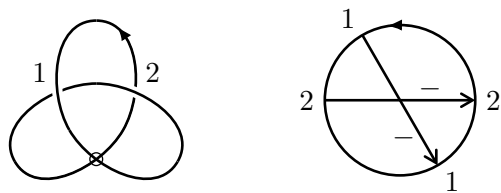
For example, Figure~\ref{Gauss1} shows the virtual trefoil knot and its corresponding Gauss diagram. The \emph{Gauss code} is simply the sequence of under and overcrossings together with the crossing signs. A Gauss code for the virtual trefoil is O1{$-$}O2{$-$}U1{$-$}U2{$-$}.

Not every Gauss diagram can be realized as the diagram of a classical knot (the diagram in Figure ~\ref{Gauss1} is an example), but every Gauss diagram {\em can} be realized as the diagram of a virtual knot \cite{MR1721925}.  It is often useful to think of virtual knots in terms of Gauss diagrams, since it avoids the pitfalls created by misinterpreting virtual knots as physical objects.
 
Note that the virtual moves V1, V2, V3 and SV have no effect on the Gauss diagram, since the sequence of classical crossings is unaffected by these moves. The classical Reidemeister moves R1, R2, R3 do change this sequence and can be reinterpreted as moves on Gauss diagrams.

Move R1 corresponds to adding or deleting a chord from the Gauss diagram whose head and tail are adjacent on the circle (i.e., no chord ends are between its head and tail). See Figure~\ref{R1Gauss}.
\begin{figure}[htbp]
\centering
\begin{tikzpicture}[line width=1]
\matrix[ampersand replacement=\&,column sep=6mm]{
\draw(-.5,.5) to[out=-90,in=-90] (.5,0);\draw[overcross] (.5,0) to[out=90,in=90] (-.5,-.5);\draw[knotarrow, shorten >=6pt] (.5,0) to[out=90,in=90] (-.5,-.5); \& 
\draw[xshift=2.5mm] (-1,.5) arc(180:315:1);\draw[chordarrow,yshift=5mm,xshift=2.5mm] (300:1) -- (195:1); \node at (.2,.1){$+$};\&[.8cm]
\draw (.5,0) to[out=90,in=90] (-.5,-.5);\draw[knotarrow, shorten >=6pt] (.5,0) to[out=90,in=90] (-.5,-.5);\draw[overcross](-.5,.5) to[out=-90,in=-90] (.5,0); \& 
\draw[xshift=2.5mm] (-1,.5) arc(180:315:1);\draw[chordarrow,yshift=5mm,xshift=2.5mm] (195:1) -- (300:1); \node at (.2,.05){$-$};\\
\node{\vlabel{R1}}; \& \node{\vlabel{R1}}; \& \node{\vlabel{R1}}; \& \node{\vlabel{R1}};\\
\draw[looseness=.8] (-.5,-.5) to[out=90,in=-90] (.5,0) to[out=90,in=-90] (-.5,.5);
\draw[looseness=.8, knotarrow, shorten >=6pt] (.5,0) to[out=-90,in=90] (-.5,-.5); \&
\draw[xshift=2.5mm](-1,.5) arc(180:315:1); \& \draw[looseness=.8] (-.5,-.5) to[out=90,in=-90] (.5,0) to[out=90,in=-90] (-.5,.5);
\draw[looseness=.8, knotarrow, shorten >=6pt] (.5,0) to[out=-90,in=90] (-.5,-.5); \&
\draw[xshift=2.5mm](-1,.5) arc(180:315:1);\\
};
\end{tikzpicture}
\caption{Reidemeister move R1 on a Gauss diagram.}
\label{R1Gauss}
\end{figure}
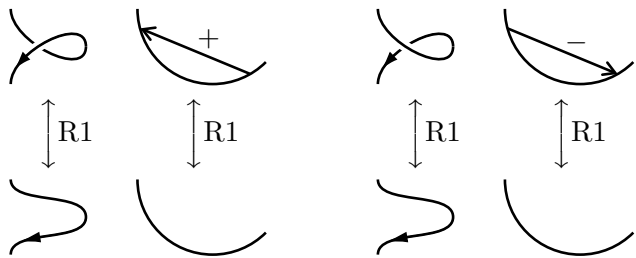
The sign and direction of the chord will depend on the direction of the twist (right-handed or left-handed) and the orientation of the knot. Specifically, changing the direction of the twist changes the sign and direction of the chord. A knot with opposite orientation would give the chord the same sign but opposite direction.

Move R2 corresponds to adding or deleting a pair of chords from the Gauss diagram that that satisfy three conditions: (1) the heads are adjacent, (2) the tails are adjacent and (3) the chords have opposite sign. See Figure~\ref{R2Gauss}.
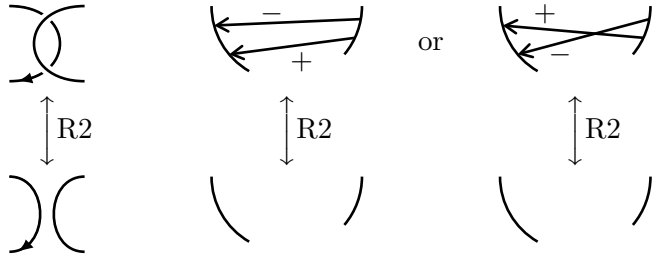
\begin{figure}[htbp]
\centering
\begin{tikzpicture}[line width=1]
\matrix[ampersand replacement=\&,column sep=6mm]{
\draw[looseness=2.3] (-.5,-.5) to[out=0,in=0] (-.5,.5);
\draw[looseness=2.3,knotarrow, shorten >=3pt] (-.5,.5) to[out=0,in=0] (-.5,-.5);
\draw[looseness=2.3,overcross] (.5,-.5) to[out=180,in=180] (.5,.5); \&[1cm] 
\draw (-1,.5) arc (180:240:1); \draw (1,.5) arc (0:-40:1); \draw[yshift=.5cm,chordarrow] (-10:1) -- (195:1); \draw[yshift=.5cm,chordarrow] (-25:1) -- (220:1); \path (-.2,.4) node{$-$} -- (.2,-.24) node{$+$}; \& \node{or}; \& \draw (-1,.5) arc (180:240:1); \draw (1,.5) arc (0:-40:1); \draw[yshift=.5cm,chordarrow] (-10:1) -- (220:1); \draw[yshift=.5cm,chordarrow] (-25:1) -- (195:1);\path (-.2,-.15) node{$-$} -- (-.4,.4) node{$+$};\\
\node{\vlabel{R2}}; \& \node{\vlabel{R2}}; \&  \& \node{\vlabel{R2}};\\
\draw[looseness=1.4] (-.5,-.5) to[out=0,in=0] (-.5,.5);
\draw[looseness=1.4,knotarrow, shorten >=3pt] (-.5,.5) to[out=0,in=0] (-.5,-.5);
\draw[looseness=1.4] (.5,-.5) to[out=180,in=180] (.5,.5); \&[1cm] \draw (-1,.5) arc (180:240:1); \draw (1,.5) arc (0:-40:1); \& \& \draw (-1,.5) arc (180:240:1); \draw (1,.5) arc (0:-40:1);\\
};
\end{tikzpicture}
\caption{Reidemeister move R2 on a Gauss diagram.}
\label{R2Gauss}
\end{figure}
The chords may or may not cross depending on the relative orientations of the two arcs of the knot. 

Move R3 is more complicated. We will use the term \emph{3-movable} to refer to a triple of chords in a Gauss diagram (or the corresponding crossings in a knot diagram) that can be repositioned by an R3 move. For example, the triple in Figure~\ref{movable}(a) is 3-movable, but the triples in Figures~\ref{movable}(b) and \ref{movable}(c) are not. Notice that a forbidden move would be required to move the ``lower'' strand in \ref{movable}(b).
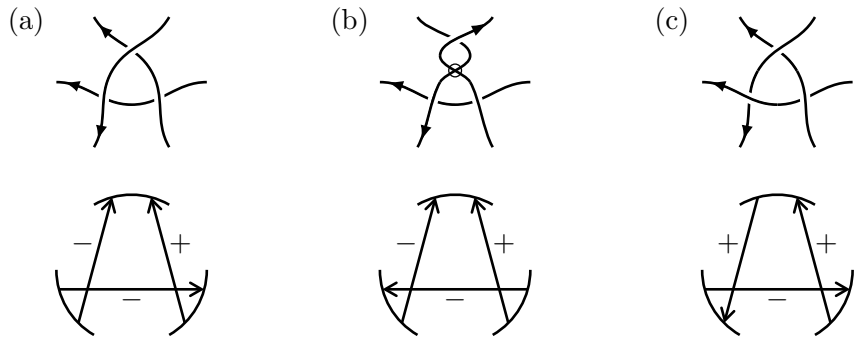
\begin{figure}[htbp]
\centering
\begin{tikzpicture}[line width=1]
\matrix[ampersand replacement=\&,column sep=1.5cm]{
\draw (1,0) to[out=180,in=0] (0,-.3) to[out=180,in=0] (-1,0); \draw[knotarrow, shorten >=3pt] (0,-.3) to[out=180,in=0] (-1,0); \draw[overcross,rotate=120] (1,0) to[out=180,in=0] (0,-.3) to[out=180,in=0] (-1,0); \draw[rotate=120, knotarrow, shorten >=3pt] (0,-.3) to[out=0,in=180] (1,0); \draw[overcross,rotate=-120] (1,0) to[out=180,in=0] (0,-.3) to[out=180,in=0] (-1,0); \draw[rotate=-120, knotarrow, shorten >=3pt] (0,-.3) to[out=0,in=180] (1,0); \node at (-1.4,.8) {(a)};\&
\draw (0:1) to[out=180,in=0] (-90:.3) to[out=180,in=0] (180:1); \draw[knotarrow,shorten >=3pt] (-90:.3) to[out=180,in=0] (180:1); \draw[looseness=.7](120:1) to[out=-60,in=90] (.2,.35) to[out=-90,in=60] (180:.2); \draw[looseness=.7,overcross] (180:.2) to[out=-120,in=60] (-120:1); \draw[knotarrow,shorten >=3pt] (180:.2) to[out=-120,in=60] (-120:1); \draw[looseness=.7,overcross] (-60:1) to[out=120,in=-60] (0:.2); \draw[looseness=.7](0:.2) to[out=120,in=-90] (-.2,.35); \draw[looseness=.7,overcross] (-.2,.35) to[out=90,in=-120] (60:1); \draw[looseness=.7,knotarrow,shorten >=3pt] (-.2,.35) to[out=90,in=-120] (60:1); \vcirc{0,.155} \node at (-1.4,.8) {(b)};\&
\draw(0:1) to[out=180,in=0] (-90:.3);\draw(150:.3) to[out=240,in=60] (240:1); \draw[knotarrow,shorten >=3pt](150:.3) to[out=240,in=60] (240:1); \draw(30:.3) to[out=120,in=-60] (120:1); \draw[knotarrow,shorten >=3pt](30:.3) to[out=120,in=-60] (120:1); \draw[overcross] (60:1) to[out=240,in=60] (150:.3); \draw[overcross] (-60:1) to[out=120,in=-60] (30:.3); \draw[overcross] (270:.3) to[out=180,in=0] (180:1); \draw[knotarrow, shorten >=3pt] (270:.3) to[out=180,in=0] (180:1); \node at (-1.4,.8) {(c)};\\[5mm]
\draw (60:1) arc(60:120:1); \draw (180:1) arc(180:240:1); \draw (300:1) arc(300:360:1); \draw[chordarrow] (195:1) -- (-15:1); \draw[chordarrow] (225:1) -- (105:1); \draw[chordarrow] (315:1) -- (75:1); \path (-.65,.35) node {$-$} -- (.65,.35) node {$+$} -- (0,-.4) node {$-$};\&
\draw (60:1) arc(60:120:1); \draw (180:1) arc(180:240:1); \draw (300:1) arc(300:360:1); \draw[chordarrow] (-15:1) -- (195:1); \draw[chordarrow] (225:1) -- (105:1); \draw[chordarrow] (315:1) -- (75:1); \path (-.65,.35) node {$-$} -- (.65,.35) node {$+$} -- (0,-.4) node {$-$};\&
\draw (60:1) arc(60:120:1); \draw (180:1) arc(180:240:1); \draw (300:1) arc(300:360:1); \draw[chordarrow] (195:1) -- (-15:1); \draw[chordarrow] (105:1) -- (225:1); \draw[chordarrow] (315:1) -- (75:1); \path (-.65,.35) node {$+$} -- (.65,.35) node {$+$} -- (0,-.4) node {$-$};\\
};
\end{tikzpicture}
\caption{Movable and nonmovable triples.}
\label{movable}
\end{figure}

A 3-movable triple must consist of three chords and three arcs of the Gauss diagram, one arc consisting of two arrowheads, one arc consisting of two arrowtails, the third arc consisting of one head and one tail of distinct chords. A triple of chords satisfying this condition is said to be \emph{matched}. The chords in Figure~\ref{movable}(a) and \ref{movable}(b) are matched, but not those in \ref{movable}(c). A second condition is necessary to guarantee that a matched triple is 3-movable. We define that condition here and then prove the correspondence with 3-movable triples.

To each chord in a matched triple we associate three numbers indicating sign, parity and direction. The sign ($\pm1$) corresponds to the sign of the crossing. The parity is assigned $+1$ if the chord crosses an even number of chords in the triple, or $-1$ for an odd number. The direction is assigned $+1$ if the chord points counterclockwise around the three arcs, or $-1$ for clockwise. For example, the horizontal chord in Figure~\ref{movable}(a) is counterclockwise ($+1$) since it points from the lower left to the lower right arc, even ($+1$) since it crosses both of the other chords, and negative ($-1$). The leftmost chord in \ref{movable}(a) is clockwise ($-1$), odd ($-1$) and negative ($-1$). Define the \emph{3-sign} of each chord to be the product of these three numbers. Note that all three chords in \ref{movable}(a) have 3-sign equal to $-1$.
\begin{thm}
A triple of chords is 3-movable if and only if it is matched and its chords all have the same 3-sign.
\end{thm}
\begin{proof}
The necessity of being matched follows directly from the definition of an R3 move, so we will show that a matched triple is 3-movable if and only if its chords have the same 3-sign. If a triple is 3-movable, we may rotate the knot so that the strand that goes under the others is horizontal and the third crossing is above, as in the bottom of Figure~\ref{movable}(a). There are two possibilities for the third crossing, \tikz[baseline, yshift=-.4ex, scale=.35,line width=1]{\draw(0,0) -- (1,1); \draw[overcross] (1,0) -- (0,1)} or \tikz[baseline, yshift=-.4ex, scale=.35,line width=1]{\draw(1,0) -- (0,1); \draw[overcross] (0,0) -- (1,1)}. For each of those, there are 2 possibilities for orientation of each strand and two orderings of the strands as we traverse the knot, resulting in $2^5=32$ movable triples.

Conversely, if we have a matched triple with all 3-signs equal, we may rotate the Gauss diagram so that the arc with both arrowheads is at the top. There are then two possible locations for the arc with one head and one tail (left or right). That arc can have the head in one of two positions (top or bottom). The head and tail on that arc each have two possible connections (two tails and two heads). Finally, the bottom (horizontal) chord can be labeled $+$ or $-$, which will determine the signs of the other two chords (since all 3-signs must be the same). Thus there are 32 possible matched triples with equal 3-signs. The correspondence with movable triples is tedious but elementary.
\end{proof}
When a triple of chords is 3-movable, the effect of the R3 move on the Gauss diagram is to switch the two chord ends on each of the three arcs as in Figure~\ref{R3Gauss}. The resulting triple is, of course, 3-movable. It is easy to see the triple is still matched, and each chord has the same direction, sign and parity as the original.
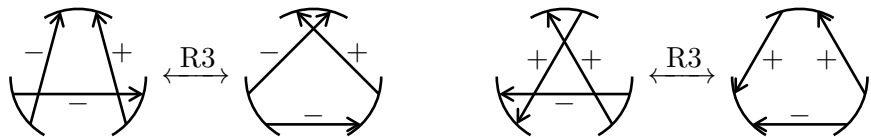
\begin{figure}[htbp]
\centering
\begin{tikzpicture}[baseline,line width=1pt,scale=.9] \draw (60:1) arc(60:120:1); \draw (180:1) arc(180:240:1); \draw (300:1) arc(300:360:1); \draw[chordarrow] (195:1) -- (-15:1); \draw[chordarrow] (225:1) -- (105:1); \draw[chordarrow] (315:1) -- (75:1); \path (-.65,.35) node {$-$} -- (.65,.35) node {$+$} -- (0,-.4) node {$-$};
\end{tikzpicture}%
\meq[.5em]{\xleftrightarrow{\mathrm{\textstyle{\hspace{1mm}R3\hspace{1mm}}}}}%
\begin{tikzpicture}[baseline,line width=1pt,scale=.9] \draw (60:1) arc(60:120:1); \draw (180:1) arc(180:240:1); \draw (300:1) arc(300:360:1); \draw[chordarrow] (225:1) -- (315:1); \draw[chordarrow] (195:1) -- (75:1); \draw[chordarrow] (-15:1) -- (105:1); \path (-.65,.35) node {$-$} -- (.65,.35) node {$+$} -- (0,-.58) node {$-$};
\end{tikzpicture}\hspace{1.5cm}%
\begin{tikzpicture}[baseline,line width=1pt,scale=.9] \draw (60:1) arc(60:120:1); \draw (180:1) arc(180:240:1); \draw (300:1) arc(300:360:1); \draw[chordarrow] (75:1) -- (225:1); \draw[chordarrow] (-15:1) -- (195:1); \draw[chordarrow] (315:1) -- (105:1); \path (-.4,.25) node {$+$} -- (.4,.25) node {$+$} -- (0,-.4) node {$-$};
\end{tikzpicture}%
\meq[.5em]{\xleftrightarrow{\mathrm{\textstyle{\hspace{1mm}R3\hspace{1mm}}}}}%
\begin{tikzpicture}[baseline,line width=1pt,scale=.9] \draw (60:1) arc(60:120:1); \draw (180:1) arc(180:240:1); \draw (300:1) arc(300:360:1); \draw[chordarrow] (105:1) -- (195:1); \draw[chordarrow] (-15:1) -- (75:1); \draw[chordarrow] (315:1) -- (225:1); \path (-.4,.25) node {$+$} -- (.4,.25) node {$+$} -- (0,-.58) node {$-$};
\end{tikzpicture}
\caption{R3 moves on Gauss diagrams.}
\label{R3Gauss}
\end{figure}

\section{Examples}
We provide two examples to illustrate the usefulness of Gauss diagrams in simplifying knots.

\begin{example}
Consider the virtual knot in Figure \ref{tough1}(a). 
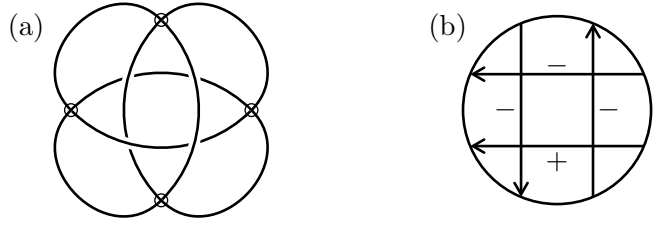
\begin{figure}[htp]
    \centering
    \raisebox{1cm}{(a)}\hspace{-5mm}
    \begin{tikzpicture}[baseline,scale=.4,line width=1]
    \draw (3,0) to[out=135,in=45] (-3,0);
    \draw[overcross] (0,3) to[out=-135,in=135] (0,-3);
    \draw[overcross,shorten >=.5pt,shorten <=.5pt] (-3,0) to[out=-45,in=-135] (3,0);
    \draw[overcross,shorten >=.5pt,shorten <=.5pt] (0,-3) to[out=45,in=-45] (0,3);
    \draw[looseness=1.5] (3,0) to[out=45,in=45] (0,3) to[out=135,in=135] (-3,0) to[out=-135,in=-135] (0,-3) to[out=-45,in=-45] (3,0);
    \vcirc{3,0}\vcirc{-3,0}\vcirc{0,3}\vcirc{0,-3}
    \end{tikzpicture}\hspace{1.5cm}
    \raisebox{1cm}{(b)}\hspace{-2mm}
    \begin{tikzpicture}[baseline,line width=1,scale=1.25]
    \draw (0,0) circle[radius=1];
    \draw[chordarrow] (112.5:1) -- (-112.5:1);
    \draw[chordarrow] (22.5:1) -- (157.5:1);
    \draw[chordarrow] (-22.5:1) -- (-157.5:1);
    \draw[chordarrow] (-67.5:1) -- (67.5:1);
    \path (.55,0) node{$-$} -- (0,.47) node{$-$} -- (-.55,0) node{$-$} -- (0,-.55) node{$+$};
    \end{tikzpicture}
    \caption{A hard-to-see unknot.}
    \label{tough1}
\end{figure}
It is difficult to see a sequence of extended Reidemeister moves to simplify the diagram; we encourage the reader to try. However, when viewed as a Gauss diagram as in \ref{tough1}(b), it is easy to see the R2 move that can be made with the two horizontal chords. Their heads and tails are both adjacent, and they have opposite sign. Thus they can be removed from the diagram, leaving two nonintersecting chords. These can be removed with R1 moves, leaving the unknot.
\end{example}
\begin{example}
Consider the knot and corresponding Gauss diagram in Figure \ref{tough2}.
\begin{figure}[htp]
    \centering
\begin{tikzpicture}[baseline,yshift=-2mm,line width=1,scale=1]
    \coordinate (A) at (-.75,-.75);
    \coordinate (B) at (.5,-1);
    \coordinate (C) at (.8,1.8);
    \coordinate (D) at (1,1.2);
    \coordinate (E) at (1,.5);
    \coordinate (F) at (-1.5,-.5);
    \draw (E) to[out=135,in=90] (F) to[out=-90,in=-135] (A);
    \draw[overcross,looseness=1.2] (A) to[out=45,in=180] (C) to[out=0,in=-45] (B) to[out=135,in=135] (D);
    \draw (D) to[out=-45,in=45] (E);
    \draw[overcross,shorten >=.5pt,shorten <=.5pt,looseness=1.5] (E) to[out=-135,in=135] (A);
    \draw (A) to[out=-45,in=-135] (B) to[out=45,in=-45] (E);
    \vcirc{A}\vcirc{B}\vcirc{E}
    \path (-.26,.96) node{\footnotesize 1} -- (-.45,.1) node{\footnotesize 2} -- (.45,.96) node{\footnotesize 3} -- (.29,-.1) node{\footnotesize 4};
\end{tikzpicture}\hspace{.5cm}
\begin{tikzpicture}[baseline,line width=1,scale=1.25]
    \draw[black!30,line width=3pt] ([shift=(60:1cm+1.5pt)]0,0) arc (60:120:1cm+1.5pt);
    \draw[black!30,line width=3pt] ([shift=(195:1cm+1.5pt)]0,0) arc (195:255:1cm+1.5pt);
    \draw[black!30,line width=3pt] ([shift=(-30:1cm+1.5pt)]0,0) arc (-30:30:1cm+1.5pt);
    \draw (0,0) circle[radius=1];
    \draw[chordarrow] (112.5:1) -- (-157.5:1);
    \draw[chordarrow] (-67.5:1) -- (157.5:1);
    \draw[chordarrow] (22.5:1) -- (-112.5:1);
    \draw[chordarrow] (-22.5:1) -- (67.5:1);
    \path (-.4,.5) node{$+$} -- (-.27,-.08) node{$-$} -- (.27,-.08) node{$+$} -- (.4,.5) node{$-$} -- (-.4,1.15) node{1} -- (-1.05,.55) node{2} -- (1.05,.55) node{3} -- (.4,1.15) node{4};
\end{tikzpicture}\vspace{-5mm}
    \caption{Setup for an R3 move.}
    \label{tough2}
\end{figure}
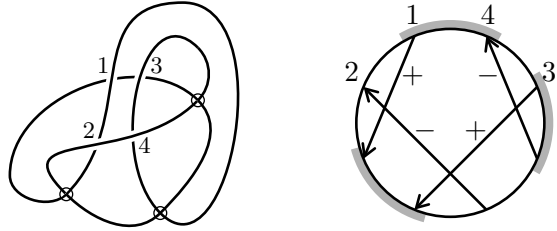
The triple of chords $(1,3,4)$ is matched (indicated by the shaded arcs) and all have positive 3-sign. Hence we may perform an R3 move to obtain the first Gauss diagram in Figure \ref{tough2a}.
\begin{figure}[htp]
    \centering
\begin{tikzpicture}[baseline,line width=1,scale=1.25]
    \draw[black!30,line width=3pt] ([shift=(60:1cm+1.5pt)]0,0) arc (60:120:1cm+1.5pt);
    \draw[black!30,line width=3pt] ([shift=(195:1cm+1.5pt)]0,0) arc (195:255:1cm+1.5pt);
    \draw[black!30,line width=3pt] ([shift=(-30:1cm+1.5pt)]0,0) arc (-30:30:1cm+1.5pt);
    \draw (0,0) circle[radius=1];
    \draw[chordarrow] (67.5:1) -- (-112.5:1);
    \draw[chordarrow] (-67.5:1) -- (157.5:1);
    \draw[chordarrow] (-22.5:1) -- (-157.5:1);
    \draw[chordarrow] (22.5:1) -- (112.5:1);
    \path (0,.35) node{$+$} -- (-.47,.12) node{$-$} -- (.45,-.21) node{$+$} -- (.6,.4) node{$-$} -- (-.4,1.15) node{4} -- (-1.05,.55) node{2} -- (1.1,-.4) node{3} -- (.4,1.15) node{1};
\end{tikzpicture}\meq[.5em]{\xleftrightarrow{\mathrm{\textstyle{\hspace{2mm}R2\hspace{2mm}}}}}
\begin{tikzpicture}[baseline,line width=1,scale=1.25]
    \draw (0,0) circle[radius=1];
    \draw[chordarrow] (67.5:1) -- (-112.5:1);
    \draw[chordarrow] (22.5:1) -- (112.5:1);
    \path (0,.35) node{$+$} -- (.6,.4) node{$-$} -- (-.4,1.15) node{4} -- (.4,1.15) node{1};
\end{tikzpicture}
    \caption{Reducing to the unknot.}
    \label{tough2a}
\end{figure}
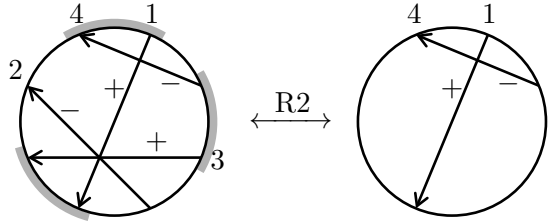
Now chords 2 and 3 are positioned for an R2 move, and can be eliminated from the diagram. Finally, chords 1 and 4 can be removed by an additional R2 move, resulting in an empty diagram, i.e., the unknot.
\end{example}

\end{document}